\documentclass[11pt]{article}
\PassOptionsToPackage{obeyspaces}{url}
\usepackage[hidelinks]{hyperref}
\usepackage{bookmark}
\bookmarksetup{open,numbered,depth=5}
\usepackage{amsfonts}
\usepackage{amsmath}
\usepackage{amssymb, amsthm, enumitem}
\usepackage{pgf,tikz}
\usepackage{bm}
\usepackage{cancel}
\usepackage{comment}
\usepackage{pgfplots}
\pgfplotsset{compat=1.15}
\usepackage{mathrsfs}
\usetikzlibrary{arrows}
\usetikzlibrary{calc}
\definecolor{ududff}{rgb}{0.30196078431372547,0.30196078431372547,1.}

\usepackage{fancyhdr}
\usepackage{etoolbox} 
\patchcmd{\thebibliography}{\leftmargin\labelwidth}{\leftmargin\labelwidth\addtolength\itemsep{-0.1\baselineskip}}{}{}

\usepackage{mathtools}
\usepackage{xstring}

\title{$4K_1$ free graphs on 13 vertices have cop number at most 2}
\date{}

\author{Zhaoyu Wu \thanks{
  School of Mathematics and Statistics, Northeastern University at Qinhuangdao, \texttt{\href{mailto:wuzhaoyuwushidi@163.com}{wuzhaoyuwushidi@163.com}}
}}

\usepackage[nameinlink]{cleveref}

\newtheorem{theorem}{Theorem}
\newtheorem{lemma}[theorem]{Lemma}

\newtheorem{conjecture}[theorem]{Conjecture}

\allowdisplaybreaks[4]

\begin{document}

\maketitle

\begin{abstract}
    The game of cops and robber has been studied for many years. Denoting $\mathsf{Forb}(4K_1)$ to be the family of all graphs that contain no induced subgraph isomorphic to $4K_1$ (e.g., with independence number less than $4$), we prove that for any $G\in\mathsf{Forb}(4K_1)$, we have $c(G)\leq 2$, where $c(\cdot)$ is the cop number. This improves a lower bound of a question proposed by Char et al. in a recent paper (arxiv, 2025), that any counterexample of a conjecture raised by Turcotte (2022) when $p=4$ must have at least 14 vertices.
\end{abstract}

\section{Basic Notations}
Throughout this paper, we consider only finite, simple graphs with undirected edges. For a positive integer $n$, let $\mathsf{Graph}(n)$ denote the family of all simple graphs on $n$ vertices.

For a graph $G$, let $\delta(G)$ and $\Delta(G)$ denote its minimum and maximum degree, respectively. We denote the domination number of $G$ by $\gamma(G)$, and denote the independence number of $G$ by $\alpha(G)$. For a vertex $v \in V(G)$, we write $d_G(v)$ for its degree in $G$, and $N_G(v)$ for its neighborhood. We define $N_G[v]$ as $N_G(v) \cup \{v\}$. We omit the subscript $G$ when the graph is clear from context. If we only consider a graph $G$ and a subgraph $H$ of $G$, we use $d(\cdot), N(\cdot), N[\cdot]$ for $G$, and $d_H(\cdot), N_H(\cdot), N_H[\cdot]$ for $H$. For a graph $G$ and a vertex set $X \subseteq V(G)$, we denote by $G - X$ the induced subgraph $G[V(G) \setminus X]$; if $X = \{v\}$, we simply write $G - v$.

For $S\subseteq V(G)$, denote by $E(S)$ the set of edges with vertices in $S$; for $S_1,S_2\subseteq V(G)$, denote by $E(S_1,S_2)$ the set of edges with vertices in $S$.

For $\ell \geq 1$, let $P_\ell$ and $K_\ell$ denote the path graph and the complete graph on $\ell$ vertices, respectively. For $\ell \geq 3$, let $C_\ell$ denote the cycle graph on $\ell$ vertices. The disjoint union of two graphs $G_1$ and $G_2$ is denoted by $G_1 + G_2$. For a graph $G$ and a positive integer $\ell$, we write $\ell G$ for the disjoint union of $\ell$ copies of $G$.

For a graph $F$, we denote by $\mathsf{Forb}(F)$ the family of graphs that contain no induced subgraph isomorphic to $F$. 






\section{Introduction}

The game of cops and robber has been a topic of study in graph theory for many years. The earliest papers on this topic are by Quilliot \cite{quilliot1978jeux} and by Nowakowski and Winkler \cite{nowakowski1983vertex}. The game is played on a connected graph $G$ by two players with the following rules:

\begin{itemize}
    \item In the first round, the first player (the cops) places $p$ cops $C_1, C_2, \dots, C_p$ on vertices of $G$. Then, the second player (the robber) places a single robber $R$ on a vertex of $G$.
    \item In each subsequent round, the first player moves first. The player decides the action for each cop $C_i$: each may either remain on its current vertex or move to an adjacent vertex. If any cop occupies the same vertex as the robber $R$ at the end of this move, the game ends and the first player wins (i.e., the cops win).
    \item Then, the second player takes a turn and decides the action for the robber $R$. Similarly, the robber may either stay or move to an adjacent vertex. The second player wins (i.e., the robber wins) if they can indefinitely prevent the cops from capturing the robber.
\end{itemize}

For a connected graph $G$, we define its cop number, denoted by $c(G)$ (this notation is due to Aigner and Fromme \cite{aigner1984game}), to be the minimum number $p$ of cops such that the cops have a winning strategy on $G$. If $G$ is disconnected, let $G_1, G_2, \dots, G_k$ be its connected components. Then, we define
\[
c(G) := \max_{1 \le i \le k} c(G_i).
\]
While this definition may seem unintuitive at first, it proves convenient in practice. For a family $\mathcal{F}$ of graphs, we define
\[
c(\mathcal{F}) = \max_{G \in \mathcal{F}} c(G).
\]

This paper focuses on upper bounds for the cop numbers of graph families. Several key results in this direction are known:

\begin{theorem}\emph{(\cite{joret2010cops})}
    For a single graph $F$, $c(\mathsf{Forb}(F))$ is finite if and only if $F$ is a finite disjoint union of paths.
\end{theorem}

\begin{theorem}\emph{(\cite{joret2010cops})}
    For any integer $t \geq 2$, $c(\mathsf{Forb}(P_t)) \leq t-2$.
\end{theorem}

\begin{theorem}\emph{(\cite{chudnovsky2024cops})}
    $c(\mathsf{Forb}(P_5)) \leq 2$.
\end{theorem}

Turcotte \cite{turcotte2022cops} raised the following interesting conjecture:

\begin{conjecture}\emph{(\cite{turcotte2022cops})}\label{big-conj}
    For any positive integer $p$, we have:
    \begin{itemize}
        \item For any graph $G$ such that $\alpha(G)<p$, then $c(G)\leq p-2$.
    \end{itemize}
\end{conjecture}

In recent work, Char, Maniya, and Pradhan \cite{char2025counterexample} established the following theorem, showing that \cref{big-conj} is false when $p=4$.

\begin{theorem}\emph{(\cite{char2025counterexample})}\label{lowerbound-12}
    Let $n$ be the largest integer such that
    \[
    c\bigl(\{G \in \mathsf{Forb}(4K_1): |V(G)| \leq n\}\bigr) \leq 2.
    \]
    Then, $12 \leq n \leq 15$.
\end{theorem}

Soon after, Clow and Zaguia \cite{clow2025cops} proved that the conclusion of \cref{big-conj} is false for any $p\geq 4$.

In the paper of Char et al. , they posed the question of determining the exact value of $n$. Our main result improves the lower bound to $n \geq 13$, which we state formally as follows.

\begin{theorem}\label{main-thm}
    Every $4K_1$-free graph $G$ on $13$ vertices satisfies $c(G) \leq 2$.
\end{theorem}

Our proof relies on the following result from their paper.

\begin{theorem}\emph{(\cite{char2025counterexample})}\label{max-n-6}
    Let $G \in \mathsf{Forb}(4K_1)$ be a graph on $n$ vertices. If $\Delta(G) \geq n-6$, then $c(G) \leq 2$.
\end{theorem}

We will also use the following well-known lemma, which appears, for example, in \cite{turcotte20214cop}.

\begin{lemma}\emph{(\cite{turcotte20214cop})}\label{corner}
    Let $G$ be a graph, and let $u$ and $v$ be vertices of $G$ such that $N(v) \subseteq N[u]$. Then, 
    \begin{itemize}
        \item if $c(G) \geq 2$, then $c(G) = c(G - v)$;
        \item if $c(G) = 1$, then $c(G - v) \in \{1, 2\}$.
    \end{itemize}
\end{lemma}

And the following simple result will also be used. Notice that the connectedness is not required

\begin{lemma}\label{add-1}
    Let $G$ be a graph and $u\in V(G)$. Then, $c(G)\leq c(G-u)+1$. 
\end{lemma}

\begin{lemma}\label{3K1-free}
    Any graph $G$ on $6$ vertices containing no induced $3K_1$ subgraph must contain one of the following as a (not necessarily induced) subgraph:
    \begin{itemize}
        \item $K_5 + K_1$;
        \item $K_4 + K_2$;
        \item $K_3 + K_3$;
        \item the blown-up cycle $B_6$ (see Figure~\ref{fig:B6}).
    \end{itemize}
    
    \begin{figure}[htbp]\label{fig:B6}
        \centering
        \begin{tikzpicture}[
            vertex/.style = {draw, circle, inner sep=1.5pt, minimum size=5mm},
            edge/.style = {thick}
        ]
            \foreach \i in {1,...,5} {
                \node[vertex] (u\i) at (90+72-\i*72:1.5) {};
            }
            \node[vertex] (v) at (0:0) {};
            
            \foreach \i [remember=\i as \last (initially 5)] in {1,...,5} {
                \draw[edge] (u\last) -- (u\i);
            }
            \foreach \i in {1,2,5} {
                \draw[edge] (v) -- (u\i);
            }
        \end{tikzpicture}
        \caption{The graph $B_6$. }\label{figure:B6}
    \end{figure}
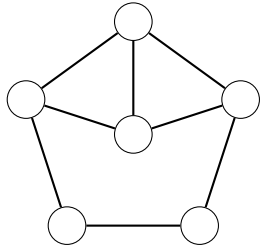
\end{lemma}

\begin{proof}
    If the complement $\overline{G}$ is bipartite, then $G$ contains one of $K_5+K_1$, $K_4+K_2$, or $K_3+K_3$. We therefore assume $\overline{G}$ is not bipartite. Since $\overline{G}$ is triangle-free, any odd cycle in it must be of length $5$ and induced. Consequently, $\overline{G}$ contains an induced $C_5$, which implies that $G$ itself has an induced subgraph $H$ isomorphic to $C_5$.
    
    Let $v$ be the unique vertex in $V(G)\setminus V(H)$. It is straightforward to verify that $d_G(v) \geq 3$. If $d_G(v) \geq 4$, then $G$ clearly contains $B_6$ as a subgraph. If $d_G(v)=3$, then $G$ is isomorphic to $B_6$ by $G\in\mathsf{Forb}(3K_1)$.
\end{proof}

In what follows, we prove \cref{main-thm}. Let $G$ be a connected graph in $\mathsf{Graph}(13) \cap \mathsf{Forb}(4K_1)$. To establish $c(G) \leq 2$, choose a vertex $u$ with $d(u) = \Delta(G)$.

If $d(u) \geq 7$, then Theorem~\ref{max-n-6} directly gives $c(G) \leq 2$. We may therefore assume $\Delta(G) = d(u) \leq 6$ for the remainder of the proof.

The proof is organized as follows. The case $d(u) = 6$ requires a detailed analysis and is treated across the next few sections. All remaining cases—namely, those with $d(u) \leq 5$ together with any sub‑cases left open from the analysis of $d(u) = 6$—are handled computationally in \cref{computation}.

\section{The induced subgraph $G-N[u]$}

We denote the induced subgraph $G - N[u]$ by $H$. By \cref{3K1-free}, it suffices to consider whether $c(G) > 2$ is possible in the following cases:
\begin{itemize}
    \item $H$ contains $K_5 + K_1$;
    \item $H$ contains $K_4 + K_2$;
    \item $H$ contains $K_3 + K_3$;
    \item $H$ contains $B_6$.
\end{itemize}

Let $C_1$ and $C_2$ be the two cops and $R$ the robber. For $P \in \{C_1, C_2, R\}$, denote the position of $P$ after the $i$-th round by $p_i(P)$. By \cref{corner}, we may assume without loss of generality that if $x, y \in V(G)$ satisfy $N(x) \subseteq N[y]$, then $x = y$. Throughout the following discussion, we work under the following set of assumptions:

\noindent ($\ast$) All of the following hold:
\begin{itemize}
    \item $c(G) \geq 3$;
    \item $G$ is connected and $4K_1$-free;
    \item $\Delta(G) \leq 6$;
    \item $c(H) \geq 2$;
    \item For all $x, y \in V(G)$ with $N(x) \subseteq N[y]$, we have $x = y$.
\end{itemize}

When discussing a position $p_i(R)$, we implicitly assume the game has not yet reached a cop-win state. Our goal in this section is to prove that $H$ must be isomorphic to either $B_6$ or the triangular prism $T_6$ (see Figure~\ref{fig:T6}). 

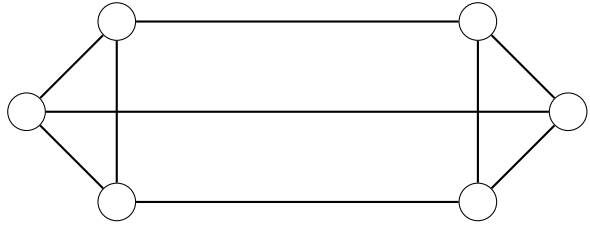
\begin{figure}[htbp]
    \centering
    \begin{tikzpicture}[
        vertex/.style = {draw, circle, inner sep=1.5pt, minimum size=5mm},
        edge/.style = {thick},
        scale=1.2
    ]
        \node[vertex] (a1) at (-2, 1) {};
        \node[vertex] (b1) at (-2, -1) {};
        \node[vertex] (c1) at (-3, 0) {};
        \node[vertex] (a2) at (2, 1) {};
        \node[vertex] (b2) at (2, -1) {};
        \node[vertex] (c2) at (3, 0) {};
        
        \draw[edge] (a1) -- (b1) -- (c1) -- (a1);
        \draw[edge] (a2) -- (b2) -- (c2) -- (a2);
        
        \draw[edge] (a1) -- node[above] {} (a2);
        \draw[edge] (b1) -- node[above] {} (b2);
        \draw[edge] (c1) -- node[above] {} (c2);
    \end{tikzpicture}
    \caption{The triangular prism $T_6$. }\label{fig:T6}
\end{figure}

The case analysis in the preceding subsections involves considerable overlap in the reasoning patterns. However, due to the nuanced variations in each specific configuration, we have chosen to present the arguments within their specific contexts rather than condensing them into a few overarching lemmas. This approach preserves the clarity of the proof for each case.

\subsection{Containing $K_5+K_1$}

By \cref{add-1} we have $c(G) \leq c(H) + 1$. If $H$ contains $K_5 + K_1$, under the assumption $c(G) \geq 3$, this forces $c(H) \geq 2$.

Now, since $H$ contains $K_5 + K_1$ as a subgraph and satisfies $\gamma(H) > 1$, the structure of $H$ is forced: $H$ must be isomorphic to $K_5 + K_1$ itself (or one cop can dominate the $K_5$ component and control the isolated vertex). This contradicts the requirement $c(H) \geq 2$. Consequently, the assumption $c(G) \geq 3$ is false, and we conclude $c(G) \leq 2$.

\subsection{Containing $K_4+K_2$}\label{K4+K2}


Let $V(H) = \{v, w, x, y, s, t\}$ such that $H[\{v, w, x, y\}] \cong K_4$ and $(s,t)$ is an edge of $H$.

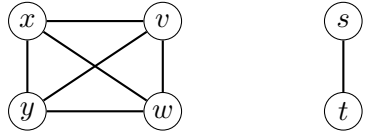
\begin{figure}[htbp]
    \centering
    \begin{tikzpicture}[
        vertex/.style = {draw, circle, inner sep=1.5pt, minimum size=5mm},
        edge/.style = {thick},
        scale=1.2
    ]
        \node[vertex] (x) at (0,1) {$x$};
        \node[vertex] (v) at (1.5,1) {$v$};
        \node[vertex] (y) at (0,0) {$y$};
        \node[vertex] (w) at (1.5,0) {$w$};
        \node[vertex] (s) at (3.5, 1) {$s$};
        \node[vertex] (t) at (3.5, 0) {$t$};
        
        \draw[edge] (v) -- (w) -- (y) -- (x) -- (v);
        \draw[edge] (v) -- (y);
        \draw[edge] (w) -- (x);
        \draw[edge] (s) -- (t);
    \end{tikzpicture}
    \caption{The graph $H$ with a $K_4$ on $\{v,w,x,y\}$ and an edge $(s,t)$.}\label{fig:H-K4-K2-initial}
\end{figure}

If all edges between $\{v,w,x,y\}$ and $\{s,t\}$ were incident to a single vertex, then $c(H)=1$, contradicting our assumption that $c(H) \ge 2$. Therefore, the edge set $E(\{v,w,x,y\},\{s,t\})$ must contain two non‑adjacent edges. Without loss of generality, we assume $(v,s), (w,t) \in E(H)$.

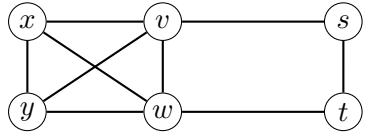
\begin{figure}[htbp]
    \centering
    \begin{tikzpicture}[
        vertex/.style = {draw, circle, inner sep=1.5pt, minimum size=5mm},
        edge/.style = {thick},
        scale=1.2
    ]
        \node[vertex] (x) at (0,1) {$x$};
        \node[vertex] (v) at (1.5,1) {$v$};
        \node[vertex] (y) at (0,0) {$y$};
        \node[vertex] (w) at (1.5,0) {$w$};
        \node[vertex] (s) at (3.5, 1) {$s$};
        \node[vertex] (t) at (3.5, 0) {$t$};
        
        \draw[edge] (v) -- (w) -- (y) -- (x) -- (v);
        \draw[edge] (v) -- (y);
        \draw[edge] (w) -- (x);
        \draw[edge] (s) -- (t);
        \draw[edge] (v) -- (s);
        \draw[edge] (w) -- (t);
    \end{tikzpicture}
    \caption{The graph $H$ with the non‑adjacent crossing edges $(v,s)$ and $(w,t)$ added.}\label{fig:H-K4-K2-crossing}
\end{figure}

Next, we prove a useful lemma.

\begin{lemma}\label{at-least-two}
    Assume that at some round $i$, we have $p_i(C_1) \in \{v,w,x,y\}$ and $p_i(C_2) \in N(u)$. Then, $p_i(R) \in \{s,t\}$ or $p_i(R) \in N(u) \cap N(s) \cap N(t)$.
\end{lemma}

\begin{proof}
    If $p_i(R) \notin \{s,t\}$, then $p_i(R) \in N(u)$. Consider two possible moves for the cops in the next round:
    \begin{itemize}
        \item Set $p_{i+1}(C_1) = v$ and $p_{i+1}(C_2) = u$. Then, the robber is forced to $p_{i+1}(R)=t$, implying $p_i(R) \in N(t)$.
        \item Set $p_{i+1}(C_1) = w$ and $p_{i+1}(C_2) = u$. Then, the robber is forced to $p_{i+1}(R)=s$, implying $p_i(R) \in N(s)$.
    \end{itemize}
    Hence $p_i(R) \in N(u) \cap N(s) \cap N(t)$.
\end{proof}

If $N(s) \cap N(u) = \emptyset$, then $N(s) \subseteq N[w]$, which contradicts our standing assumption ($\ast$). We now show that $|N(u) \cap N(s) \cap N(t)| \ge 2$. Suppose, for a contradiction, that $|N(u) \cap N(s) \cap N(t)| \le 1$. Place the cops initially as $p_1(C_1)=w$ and choose $p_1(C_2) \in N(s) \cap N(u)$ such that $N(u) \cap N(s) \cap N(t) \subseteq \{p_1(C_2)\}$. Then, $p_1(R) \notin \{s,t\}$, and by Lemma~\ref{at-least-two} we must have $p_1(R) \in N(u) \cap N(s) \cap N(t)$. Consequently $p_1(R) = p_1(C_2)$, an immediate capture, contradicting the assumption that the robber can avoid capture. Therefore $|N(u) \cap N(s) \cap N(t)| \ge 2$.

Notice that $|N(u)\cap N(s)\cap N(t)|\leq |N(s)|-|N(s)\cap H|\leq 4$. We proceed by considering different cases depending on the value of $|N(u)\cap N(s)\cap N(t)|$.

\subsubsection{The case $|N(u)\cap N(s)\cap N(t)|=2$}
Assume that $|N(u) \cap N(s) \cap N(t)| = 2$, and let $\{a, b\} = N(u) \cap N(s) \cap N(t)$. 

If $N(a) \cap \{v,w,x,y\} \neq \emptyset$, we may place a cop $p_1(C_1)$ in $N(a) \cap \{v,w,x,y\}$ and set $p_1(C_2) = b$. Applying Lemma~\ref{at-least-two} then yields a contradiction. Hence, $N(a) \cap \{v,w,x,y\} = \emptyset$, and by symmetry, $N(b) \cap \{v,w,x,y\} = \emptyset$ as well.

Now, set $p_1(C_1) = t$ and $p_1(C_2) = u$. This forces $p_1(R) \in \{v,x,y\}$. In the next round, move to $p_2(C_1) = w$ and $p_2(C_2) = a$. By Lemma~\ref{at-least-two}, we must have $p_2(R) \in \{s,t,a,b\}$. The only feasible possibility is $p_2(R) = b$. However, since $N[b] \cap \{v,x,y\} = \emptyset$, the robber cannot have moved from $\{v,x,y\}$ to $b$, a contradiction.

\subsubsection{The case $|N(u)\cap N(s)\cap N(t)|=3$}
Assume that $|N(u) \cap N(s) \cap N(t)| = 3$. 

First, suppose there exist at least two edges inside $N(u) \cap N(s) \cap N(t)$. Then, we can choose $p_1(C_1) = v$ and select $p_1(C_2) \in N(u) \cap N(s) \cap N(t)$ such that $N(u) \cap N(s) \cap N(t) \subseteq N[p_1(C_2)]$. By Lemma~\ref{at-least-two}, this leads to a contradiction because $(N(u) \cap N(s) \cap N(t)) \cup \{s,t\} \subseteq N[p_1(C_1)] \cup N[p_1(C_2)]$ would force an immediate capture. Therefore, there can be at most one edge within $N(u) \cap N(s) \cap N(t)$.

Next, consider the case with exactly one such edge. Write $N(u) \cap N(s) \cap N(t) = \{a, b, c\}$, where $(b,c) \in E(G)$ and $(a,b), (a,c) \notin E(G)$. If $N(a) \cap \{v,w,x,y\} \neq \emptyset$, placing $p_1(C_1)$ in this intersection and setting $p_1(C_2)=b$ gives $\{s,t\} \cup (N(u) \cap N(s) \cap N(t)) \subseteq N[p_1(C_1)] \cup N[p_1(C_2)]$, which again contradicts Lemma~\ref{at-least-two}. Thus, $N(a) \cap \{v,w,x,y\} = \emptyset$.

Now, start with $p_1(C_1)=s$ and $p_1(C_2)=u$, forcing $p_1(R) \in \{w,x,y\}$. Then, set $p_2(C_1)=v$ and $p_2(C_2)=b$. By Lemma~\ref{at-least-two}, we have $p_2(R) \in \{s,t,a,b,c\}$, which implies $p_2(R)=a$. Consequently, $p_1(R) \in N(a) \cap \{v,w,x,y\}$, contradicting the earlier conclusion that this intersection is empty.

Finally, suppose there are no edges inside $N(u) \cap N(s) \cap N(t)$. We claim that there cannot exist two distinct vertices $p,q \in \{v,w,x,y\}$ and two distinct vertices $a,b \in N(u) \cap N(s) \cap N(t)$ such that $(p,a), (q,b) \in E(G)$. Indeed, if such a configuration existed, let $N(u) \cap N(s) \cap N(t) = \{a,b,c\}$ and set $p_1(C_1)=p$ and $p_1(C_2)=c$. By Lemma~\ref{at-least-two}, we get $p_1(R) \in \{s,t,a,b,c\}$, so necessarily $p_1(R)=b$. Then, taking $p_2(C_1)=q$ and $p_2(C_2)=c$ forces $p_2(R)=a$. But since $(a,b) \notin E(G)$, the robber cannot move from $b$ to $a$, a contradiction.

Because $G$ is $4K_1$-free, every vertex in $\{v,w,x,y\}$ must be adjacent to at least one vertex in $N(u) \cap N(s) \cap N(t)$. By the claim above, they must all be adjacent to the *same* vertex, say $a$. This implies $d(a) \ge 1 + |V(H)| = 7$, contradicting the assumption $\Delta(G) \le 6$.

\subsubsection{The case $|N(u)\cap N(s)\cap N(t)|=4$}
In this case, we deduce that $N(s) = \{v,t\} \cup (N(u) \cap N(s) \cap N(t))$ and $N(t) = \{w,s\} \cup (N(u) \cap N(s) \cap N(t))$.

Begin with $p_1(C_1)=w$ and $p_1(C_2)=u$, which forces $p_1(R)=s$. If $N(u) \cap N(v) \neq \emptyset$, we may set $p_2(C_1)=t$ and choose $p_2(C_2) \in N(u) \cap N(v)$, resulting in the capture of $R$—a contradiction. Hence, $N(u) \cap N(v) = \emptyset$.

Now, set $p_2(C_1)=t$ and $p_2(C_2) \in N(u) \cap N(s) \cap N(t)$. This forces $p_2(R)=v$. In the following round, take $p_3(C_1)=w$ and $p_3(C_2)=s$. The robber is now captured, a final contradiction.

\subsection{Containing $K_3+K_3$}
Let $S_1$ and $S_2$ be two disjoint subsets of $V(H)$ with $|S_1|=|S_2|=3$, such that $G[S_1]$ and $G[S_2]$ are both isomorphic to $K_3$. We consider cases based on the size of the edge cut $|E(S_1, S_2)|$.

\subsubsection{The case $|E(S_1,S_2)|\leq 2$}
If $|E(S_1, S_2)| \le 2$, the structure of $H$ is forced. Any other configuration would imply $c(H)=1$, contradicting our assumptions. The only remaining possibility is that $H$ is isomorphic to the graph in Figure~\ref{fig:K3K3-small-cut}, where $S_1 = \{s, v, x\}$, $S_2 = \{w, y, t\}$, and $E(S_1, S_2) = \{(v,w), (x,y)\}$.

\begin{figure}[htbp]
    \centering
    \begin{tikzpicture}[
        vertex/.style = {draw, circle, inner sep=1.5pt, minimum size=5mm},
        edge/.style = {thick},
        scale=0.9
    ]
        \node[vertex] (s) at (0, -1) {$s$};
        \node[vertex] (v) at (1.5, 0.5) {$v$};
        \node[vertex] (x) at (1.5, -2.5) {$x$};
        \node[vertex] (w) at (4, 0.5) {$w$};
        \node[vertex] (t) at (5.5, -1) {$t$};
        \node[vertex] (y) at (4, -2.5) {$y$};
        
        \draw[edge] (s) -- (v) -- (x) -- (s);
        \draw[edge] (w) -- (t) -- (y) -- (w);
        \draw[edge] (v) -- (w);
        \draw[edge] (x) -- (y);
    \end{tikzpicture}
    \caption{The graph $H$ when $H[S_1] \cong H[S_2] \cong K_3$ and $|E(S_1, S_2)| = 2$.}\label{fig:K3K3-small-cut}
\end{figure}
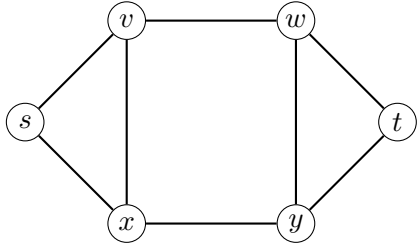

We first prove that $N(y) \cap N(u) \neq \emptyset$. Suppose, for a contradiction, that $N(y) \cap N(u) = \emptyset$. Place the cops as $p_1(C_1)=v$ and $p_1(C_2)=u$. For the robber to avoid immediate capture, we must have $p_1(R) \in \{y, t\}$.
If $p_1(R) = t$, then we take $p_2(C_1)=w$ while keeping $p_2(C_2)=u$ captures the robber. Thus, $p_1(R) = y$.
Now, keep $p_2(C_1)=v$ and choose $p_2(C_2)$ from the non-empty set $N(u) \cap N(H)$ (non-emptiness follows from the connectivity of $G$). Then, $p_2(R) \in \{y, t\}$. If $p_2(R)=t$, we can capture as before by setting $p_3(C_1)=w$ and $p_3(C_2)=u$. If $p_2(R)=y$, we take $p_3(C_2)$ to be a neighbor of $p_2(C_2)$ in $H$ and set $p_3(C_1) \in \{v, w\}$ so that the two cops lie in different sets between $S_1$ and $S_2$. Since $N(y) \cap N(u) = \emptyset$, the robber has no safe move and is captured---a contradiction. Hence, $N(y) \cap N(u) \neq \emptyset$.

Next, we show that $|(N(u) \cap N(w) \cap N(x) \cap N(y)) \setminus N(v)| \ge 2$. Assume the contrary. Choose $p_1(C_1)=v$ and select $p_1(C_2) \in N(u) \cap N(y)$ such that $(N(u) \cap N(w) \cap N(x) \cap N(y)) \setminus N(v) \subseteq \{p_1(C_2)\}$.
Then, $p_1(R) \in N(u) \setminus N(v)$; indeed, if $p_1(R)=t$, placing $p_2(C_1)=w$ and $p_2(C_2)=u$ would capture the robber. We now examine three distinct cop strategies:
\begin{itemize}
    \item Set $p_2(C_1)=v$ and $p_2(C_2)=u$. Then, $p_2(R) \in \{y,t\}$. If $p_2(R)=t$, we capture by moving $C_1$ to $w$; therefore $p_2(R)=y$, implying $p_1(R) \in N(y)$.
    \item Set $p_2(C_1)=x$ and $p_2(C_2)=u$. Then, $p_2(R) \in \{w,t\}$. If $p_2(R)=t$, we capture by moving $C_1$ to $y$; therefore $p_2(R)=w$, implying $p_1(R) \in N(w)$.
    \item Set $p_2(C_1)=w$ and $p_2(C_2)=u$. Then, $p_2(R) \in \{x,s\}$. If $p_2(R)=s$, we capture by moving $C_1$ to $v$; therefore $p_2(R)=x$, implying $p_1(R) \in N(x)$.
\end{itemize}
Combining these, we get $p_1(R) \in (N(u) \cap N(w) \cap N(x) \cap N(y)) \setminus N(v)$, which forces $p_1(R) = p_1(C_2)$, an immediate contradiction.

By symmetry, the following lower bounds also hold:
\[
\begin{aligned}
&|(N(u) \cap N(w) \cap N(x) \cap N(v)) \setminus N(y)| \ge 2, \\
&|(N(u) \cap N(w) \cap N(v) \cap N(y)) \setminus N(x)| \ge 2, \\
&|(N(u) \cap N(v) \cap N(x) \cap N(y)) \setminus N(w)| \ge 2.
\end{aligned}
\]
These four sets are disjoint subsets of $N(u)$, which would imply $|N(u)| \ge 8$, contradicting the degree condition $\Delta(G) \le 6$. This completes the analysis for $|E(S_1,S_2)| \le 2$.


\subsubsection{The case $|E(S_1,S_2)|=3$}
When $|E(S_1, S_2)| = 3$, we can check that the only possibilities for which $c(H) \ge 2$ are when $H$ is isomorphic to the triangular prism $T_6$ (see Figure~\ref{fig:T6}) or to the graph depicted in Figure~\ref{fig:K3K3-three-edges}. We now focus on the latter case.

Let $S_1 = \{v, x, y\}$ and $S_2 = \{s, t, z\}$, with the crossing edges $E(S_1, S_2) = \{(v,s), (v,t), (y,z)\}$.

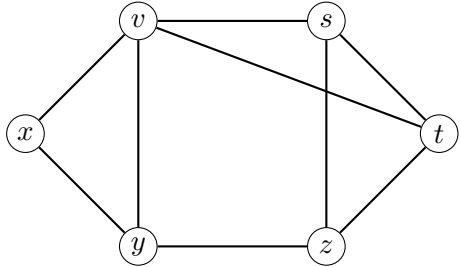
\begin{figure}[htbp]
    \centering
    \begin{tikzpicture}[
        vertex/.style = {draw, circle, inner sep=1.5pt, minimum size=5mm},
        edge/.style = {thick},
        scale=1.0
    ]
       \node[vertex] (x) at (0, -1) {$x$};
        \node[vertex] (v) at (1.5, 0.5) {$v$};
        \node[vertex] (y) at (1.5, -2.5) {$y$};
        \node[vertex] (s) at (4, 0.5) {$s$};
        \node[vertex] (t) at (5.5, -1) {$t$};
        \node[vertex] (z) at (4, -2.5) {$z$};
        
        \draw[edge] (v) -- (x) -- (y) -- (v);
        \draw[edge] (s) -- (t) -- (z) -- (s);
        \draw[edge] (v) -- (s);
        \draw[edge] (v) -- (t);
        \draw[edge] (y) -- (z);
    \end{tikzpicture}
    \caption{A second graph $H$ with $|E(S_1, S_2)| = 3$ and $c(H) \ge 2$.}\label{fig:K3K3-three-edges}
\end{figure}

First, we prove that $N(y) \cap N(u) \neq \emptyset$. Suppose otherwise. Place the cops as $p_1(C_1)=s$ and $p_1(C_2)=u$. Then, the robber must start in $p_1(R) \in \{x, y\}$ to avoid immediate capture.
\begin{itemize}
    \item If $p_1(R)=x$, set $p_2(C_1)=v$ and $p_2(C_2)=u$. The robber is now trapped and captured.
    \item If $p_1(R)=y$, keep $p_2(C_1)=s$ and choose $p_2(C_2)$ from the non-empty set $N(u) \cap N(H)$. The robber must remain at $y$; moving to $x$ would allow capture by moving $C_1$ to and $C_2$ to $v$. In the next round, move $p_3(C_2)$ to a neighbor in $H$ of $p_2(C_2)$ and set $p_3(C_1) \in \{v, s\}$ so that the two cops lie in different sets from $S_1$ and $S_2$. With $N(y) \cap N(u) = \emptyset$, the robber has no safe move and is captured.
\end{itemize}
Both possibilities lead to a contradiction, so $N(y) \cap N(u)$ is non-empty.

Next, we claim and prove that $|N(u) \cap N(v) \cap N(y) \cap N(z)| \ge 2$. Assume, for contradiction, that this intersection contains at most one vertex. Choose $p_1(C_1)=s$ and select $p_1(C_2) \in N(u) \cap N(y)$ such that $N(u) \cap N(v) \cap N(y) \cap N(z) \subseteq \{p_1(C_2)\}$.
Then, the robber must choose $p_1(R) \in N(u) \setminus N(s)$. (If $p_1(R)=x$, setting $p_2(C_1)=v$ and $p_2(C_2)=u$ captures the robber.) We analyze three distinct cop strategies:
\begin{itemize}
    \item Set $p_2(C_1)=v$ and $p_2(C_2)=u$. This forces $p_2(R)=z$, implying $p_1(R) \in N(z)$.
    \item Set $p_2(C_1)=s$ and $p_2(C_2)=u$. Then, $p_2(R) \in \{x, y\}$. If $p_2(R)=x$, moving $p_3(C_1)=v$ captures the robber; thus $p_2(R)=y$, implying $p_1(R) \in N(y)$.
    \item Set $p_2(C_1)=z$ and $p_2(C_2)=u$. Then, $p_2(R) \in \{x, v\}$. If $p_2(R)=x$, moving $p_3(C_1)=y$ captures the robber; thus $p_2(R)=v$, implying $p_1(R) \in N(v)$.
\end{itemize}
Combining these implications, we deduce that $p_1(R) \in N(u) \cap N(v) \cap N(y) \cap N(z)$. By our initial choice of $p_1(C_2)$, this forces $p_1(R) = p_1(C_2)$, an immediate capture—a contradiction. Therefore, $|N(u) \cap N(v) \cap N(y) \cap N(z)| \ge 2$.

Since $\Delta(G) \le 6$ and $v$ is adjacent to $x, y, s, t$, the set $N(u) \cap N(v) \cap N(y) \cap N(z)$ can have size at most $2$. Consequently, we have equality: $|N(u)\cap N(v)|=|N(u) \cap N(v) \cap N(y) \cap N(z)| = 2$.

Finally, we derive a contradiction from this precise count. Place $p_1(C_1)=z$ and $p_1(C_2) \in N(u) \cap N(v) \cap N(y) \cap N(z)$. The robber cannot choose $p_1(R)=x$, as then $p_2(C_1)=y$ and $p_2(C_2)=u$ would lead to capture. Hence, $p_1(R) \in N(u) \setminus N(z)\subseteq N(u)\setminus N(v)$.
Now, set $p_2(C_1)=z$ and $p_2(C_2)=u$. This forces $p_2(R)=x$. In the next round, setting $p_3(C_1)=y$ and $p_3(C_2)=u$ captures the robber—a final contradiction. This completes the analysis for this graph structure, leaving $T_6$ as the only remaining candidate when $|E(S_1,S_2)| = 3$.



\subsubsection{The case $|E(S_1,S_2)|=4$, case 1}
When $|E(S_1, S_2)| = 4$, there are two possible graph structures for $H$ that satisfy $c(H) \ge 2$. The first candidate is depicted in Figure~\ref{fig:K3K3-four-edges-case1}.

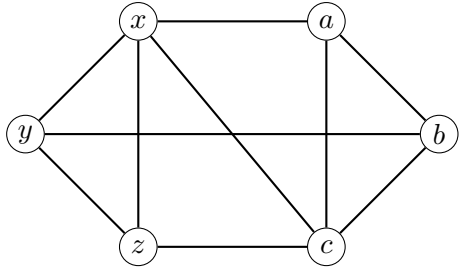
\begin{figure}[htbp]
    \centering
    \begin{tikzpicture}[
        vertex/.style = {draw, circle, inner sep=1.5pt, minimum size=5mm},
        edge/.style = {thick},
        scale=1.0
    ]

        \node[vertex] (y) at (0, -1) {$y$};
        \node[vertex] (x) at (1.5, 0.5) {$x$};
        \node[vertex] (z) at (1.5, -2.5) {$z$};
        \node[vertex] (a) at (4, 0.5) {$a$};
        \node[vertex] (b) at (5.5, -1) {$b$};
        \node[vertex] (c) at (4, -2.5) {$c$};
        
        \draw[edge] (y) -- (x) -- (z) -- (y);
        \draw[edge] (a) -- (b) -- (c) -- (a);
        \draw[edge] (x) -- (a);
        \draw[edge] (x) -- (c);
        \draw[edge] (z) -- (c);
        \draw[edge] (y) -- (b);
    \end{tikzpicture}
    \caption{The first candidate for $H$ when $|E(S_1, S_2)| = 4$. Here $S_1 = \{x, y, z\}$ and $S_2 = \{a, b, c\}$.}\label{fig:K3K3-four-edges-case1}
\end{figure}

We first prove that $N(b) \cap N(u) \neq \emptyset$. Assume the contrary, i.e., $N(b) \cap N(u) = \emptyset$. Place the cops as $p_1(C_1)=x$ and $p_1(C_2)=u$. This forces the robber to choose $p_1(R)=b$ to avoid immediate capture. For all subsequent rounds, keep $C_1$ fixed at $x$ while $C_2$ gradually moves along a shortest path in $G$ towards $b$. Since $b$ has no neighbor in $N(u)$, the robber is trapped at $b$ throughout this process and will eventually be captured by $C_2$—a contradiction. Hence, $N(b) \cap N(u)$ is non-empty.

Next, we claim that there exist at least two vertices in $N(u)$ that are adjacent to $b$ and $y$, are **not** adjacent to $x$, and are adjacent to at least one of $\{a, c\}$. Suppose this claim is false. Then, we can choose $p_1(C_1)=x$ and select $p_1(C_2) \in N(b) \cap N(u)$ such that
\[
\bigl(N(u) \cap N(b) \cap N(y) \cap (N(a) \cup N(c))\bigr) \setminus N(x) \subseteq \{p_1(C_2)\}.
\]
Now consider three distinct cop strategies:
\begin{itemize}
    \item Set $p_2(C_1)=x$ and $p_2(C_2)=u$. This forces $p_2(R)=b$, implying $p_1(R) \in N(b)$.
    \item Set $p_2(C_1)=c$ and $p_2(C_2)=u$. This forces $p_2(R)=y$, implying $p_1(R) \in N(y)$.
    \item Set $p_2(C_1)=y$ and $p_2(C_2)=u$. This forces $p_2(R) \in \{a, c\}$, implying $p_1(R) \in N(a) \cup N(c)$.
\end{itemize}
Combining these implications gives $p_1(R) \in \bigl(N(u) \cap N(b) \cap N(y) \cap (N(a) \cup N(c))\bigr) \setminus N(x)$, which forces $p_1(R) = p_1(C_2)$—an immediate contradiction. Therefore, the claim holds.

By a symmetric argument, there also exist at least two vertices in $N(u)$ that are adjacent to $b$ and $y$, are **not** adjacent to $c$, and are adjacent to at least one of $\{x, z\}$. Since $d(b) \le 6$ and $b$ is already adjacent to $\{x, y, a, c\}$, we have $|N(b) \cap N(u)| \le 3$. Consequently, there must exist a vertex $w \in N(u)$ that satisfies **both** of the above conditions simultaneously. For such a $w$, its neighborhood within $H$ is precisely $N(w) \cap H = \{y, z, a, b\}$.

Let
\[
S := \bigl(N(u) \cap N(b) \cap N(y) \cap (N(a) \cup N(c))\bigr) \setminus N(z).
\]
We assert that $|S| \ge 2$, and if equality holds, the two vertices in $S$ are non‑adjacent. Suppose this assertion is false. Then, we can place $p_1(C_1)=z$ and choose $p_1(C_2) \in N(u) \cap N(b)$ such that $S \subseteq N[p_1(C_2)]$.
Under this placement, the robber must choose $p_1(R) \in N(u) \cup \{a\}$. If $p_1(R)=a$, setting $p_2(C_1)=c$ and $p_2(C_2)=u$ captures the robber immediately. Hence, $p_1(R) \in N(u)$. We again examine three cop strategies:
\begin{itemize}
    \item $p_2(C_1)=x$ and $p_2(C_2)=u$ forces $p_2(R)=b$, so $p_1(R) \in N(b)$.
    \item $p_2(C_1)=c$ and $p_2(C_2)=u$ forces $p_2(R)=y$, so $p_1(R) \in N(y)$.
    \item $p_2(C_1)=y$ and $p_2(C_2)=u$ forces $p_2(R) \in \{a, c\}$, so $p_1(R) \in N(a) \cup N(c)$.
\end{itemize}
Together these imply $p_1(R) \in S \subseteq N[p_1(C_2)]$, again leading to a contradiction. Therefore, $|S| \ge 2$ and, if $|S|=2$, the two vertices in $S$ are not adjacent.

Notice that $|N(u) \cap N(b) \cap N(y)| \le 3$ (again by the degree bound on $b$). Since $S \subseteq N(u) \cap N(b) \cap N(y)$ and $|S| \ge 2$, we actually have $|N(u) \cap N(b) \cap N(y)| = 3$(recall that $w\in N(u) \cap N(b) \cap N(y)$). Let these three vertices be $w$, $s$, and $t$, where $\{s, t\} = S$ and $(s, t) \notin E(G)$. By symmetry, $\{s, t\}$ also equals the set $\bigl(N(u) \cap N(b) \cap N(y) \cap (N(x) \cup N(z))\bigr) \setminus N(a)$. Consequently, $N(s) \cap H = N(t) \cap H = \{x, y, b, c\}$.

Finally, observe that the set $\{s, t, a, z\}$ induces an independent set in $G$. Thus, $G[\{s,t,a,z\}] \cong 4K_1$, contradicting the hypothesis that $G$ is $4K_1$-free. This eliminates the first candidate graph.




\subsubsection{The case $|E(S_1,S_2)|=4$, case 2}
The second and final candidate for $H$ when $|E(S_1, S_2)| = 4$ is shown in Figure~\ref{fig:K3K3-four-edges-case2}. It consists of two triangles $S_1 = \{v, x, y\}$ and $S_2 = \{s, t, w\}$, with the four crossing edges $\{(v,s), (v,t), (w,x), (w,y)\}$, plus the additional edges $(v,w)$ and $(x,t)$.

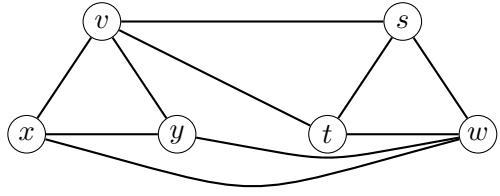
\begin{figure}[htbp]
    \centering
    \begin{tikzpicture}[
        vertex/.style = {draw, circle, inner sep=1.5pt, minimum size=5mm},
        edge/.style = {thick},
        curvededge/.style = {thick, looseness=1.5, bend right=15}
    ]
        \node[vertex] (v) at (0, 0) {$v$};
        \node[vertex] (x) at (-1, -1.5) {$x$};
        \node[vertex] (y) at (1, -1.5) {$y$};
        \node[vertex] (s) at (4, 0) {$s$};
        \node[vertex] (t) at (3, -1.5) {$t$};
        \node[vertex] (w) at (5, -1.5) {$w$};
        
        \draw[edge] (v) -- (x) -- (y) -- (v);
        \draw[edge] (s) -- (t) -- (w) -- (s);
        \draw[edge] (v) -- (s);
        \draw[edge] (v) -- (t);
        \draw[curvededge, bend right=10] (y) to (w);
        \draw[curvededge, bend right=15] (x) to (w);
    \end{tikzpicture}
    \caption{The second candidate for $H$ when $|E(S_1, S_2)| = 4$.}\label{fig:K3K3-four-edges-case2}
\end{figure}

We first show that $N(v) \cap N(u) \neq \emptyset$. Suppose, for a contradiction, that $N(v) \cap N(u) = \emptyset$. Place the cops as $p_1(C_1)=w$ and $p_1(C_2)=u$. This forces the robber to choose $p_1(R)=v$ to avoid capture. Thereafter, keep $C_1$ fixed at $w$ while $C_2$ moves along a shortest path in $G$ towards $v$. Since $v$ has no neighbor in $N(u)$, the robber remains trapped at $v$ and will eventually be captured by $C_2$—a contradiction. Hence, $N(v) \cap N(u)$ is non‑empty.

Next, we establish a lower bound on a specific intersection. Let
\[
T := \bigl(N(u) \cap N(v) \cap (N(x)\cup N(y)) \cap (N(s)\cup N(t)) \bigr) \setminus N(w).
\]
We claim that $|T| \ge 2$. Assume the contrary. Then, we can choose $p_1(C_1)=w$ and select $p_1(C_2) \in N(u) \cap N(v)$ such that $T \subseteq \{p_1(C_2)\}$. Consequently, the robber must choose $p_1(R) \in N(u) \setminus N(w)$. We now examine three different cop strategies:
\begin{itemize}
    \item Set $p_2(C_1)=w$ and $p_2(C_2)=u$. This forces $p_2(R)=v$, implying $p_1(R) \in N(v)$.
    \item Set $p_2(C_1)=s$ and $p_2(C_2)=u$. This forces $p_2(R) \in \{x, y\}$, implying $p_1(R) \in N(x) \cup N(y)$.
    \item Set $p_2(C_1)=x$ and $p_2(C_2)=u$. This forces $p_2(R) \in \{s, t\}$, implying $p_1(R) \in N(s) \cup N(t)$.
\end{itemize}
Combining these implications yields $p_1(R) \in T$, which forces $p_1(R) = p_1(C_2)$—an immediate contradiction. Therefore, $|T| \ge 2$. Since $d(v) \le 6$ and $v$ is already adjacent to $\{x, y, s, t, w\}$, we have $|N(v) \cap N(u)| \le 2$. Hence, $|N(v) \cap N(u)| = 2$ and $T = N(v) \cap N(u)$.

By a symmetric argument (exchanging the roles of $v$ and $w$), we also obtain $|N(w) \cap N(u)| = 2$ and the analogous set $\bigl(N(u) \cap N(w) \cap (N(x)\cup N(y)) \cap (N(s)\cup N(t)) \bigr) \setminus N(v)$ has size $2$. It follows that $N(v) \cap N(w) \cap N(u) = \emptyset$.

Now, choose initial cop positions $p_1(C_1) \in N(u) \cap N(v)$ and $p_1(C_2) \in N(u) \cap N(w)$. We analyze the robber's possible response $p_1(R)$. If $p_1(R) \in \{s, t, x, y\}$, observe that within $H$ we have $N_H[s] = N_H[t]$ and $N_H[x] = N_H[y]$. Therefore, there exists a vertex $a \in \{x, y, s, t\}$ such that $a$ is adjacent to $p_1(C_1)$ and satisfies $N_H[a] = N_H[p_1(R)]$. Setting $p_2(C_1)=a$ and $p_2(C_2)=u$ would then capture the robber—a contradiction. Hence, $p_1(R) \notin \{s, t, x, y\}$, and we must have $p_1(R) \in N(u)$.

Finally, consider two alternative moves for the cops in the next round:
\begin{itemize}
    \item Set $p_2(C_1)=v$ and $p_2(C_2)=u$. This forces $p_2(R)=w$, implying $p_1(R) \in N(w)$.
    \item Set $p_2(C_1)=u$ and $p_2(C_2)=w$. This forces $p_2(R)=v$, implying $p_1(R) \in N(v)$.
\end{itemize}
Together these imply $p_1(R) \in N(v) \cap N(w) \cap N(u)$. However, we previously established that $N(v) \cap N(w) \cap N(u) = \emptyset$—a final contradiction. This eliminates the second candidate graph, completing the analysis for $|E(S_1, S_2)| = 4$.

\subsubsection{The case $|E(S_1,S_2)| \ge 5$}
We now consider the case where there are at least five edges between the two $K_3$ subgraphs. By discussions in \cref{K4+K2}, we can exclude all cases when $H$ contains $K_4+K_2$. We can easily check that if $|E(S_1, S_2)| \ge 5$ and $c(H) \ge 2$, but $H$ does not contain $K_4+K_2$ as a subgraph, then the structure is forced: we must have $|E(S_1, S_2)| = 6$, and $H$ is isomorphic to the graph depicted in Figure~\ref{fig:K3K3-six-edges}. (All configurations with exactly five edges either satisfy $c(H)=1$ or contain a $K_4+K_2$, and are therefore excluded under our assumptions.)

This graph can be described as follows: take two disjoint triangles $S_1 = \{y, x, z\}$ and $S_2 = \{a, b, c\}$. In addition to the six edges within the triangles, the set $E(S_1, S_2)$ consists of precisely those six edges that connect every vertex in $S_1$ to exactly two vertices in $S_2$ (and vice versa).

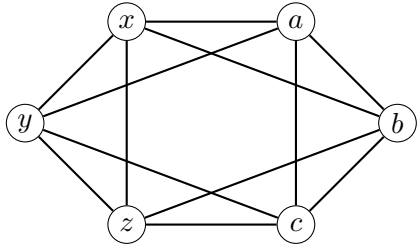
\begin{figure}[htbp]
    \centering
    \begin{tikzpicture}[
        vertex/.style = {draw, circle, inner sep=1.5pt, minimum size=5mm},
        edge/.style = {thick},
        scale=0.9
    ]
        \node[vertex] (y) at (0, -1) {$y$};
        \node[vertex] (x) at (1.5, 0.5) {$x$};
        \node[vertex] (z) at (1.5, -2.5) {$z$};
        \node[vertex] (a) at (4, 0.5) {$a$};
        \node[vertex] (b) at (5.5, -1) {$b$};
        \node[vertex] (c) at (4, -2.5) {$c$};
        
        \draw[edge] (x) -- (y) -- (z) -- (x);
        \draw[edge] (a) -- (b) -- (c) -- (a);
        \draw[edge] (y) -- (a); 
        \draw[edge] (y) -- (c); 
        \draw[edge] (z) -- (c); 
        \draw[edge] (z) -- (b); 
        \draw[edge] (x) -- (a); 
        \draw[edge] (x) -- (b); 
    \end{tikzpicture}
    \caption{The graph $H$ with $|E(S_1, S_2)| = 6$.}\label{fig:K3K3-six-edges}
\end{figure}

We first show that every vertex of $H$ must have a neighbor in $N(u)$. Suppose, for contradiction, that some vertex—say $x$—satisfies $N(x) \cap N(u) = \emptyset$. Then, the following cop strategy guarantees capture: place $p_1(C_1)=c$ and $p_1(C_2)=u$. This forces $p_1(R)=x$. Thereafter, keep $C_1$ fixed at $c$ while $C_2$ moves along a shortest path in $G$ towards $x$. Since $x$ has no neighbor in $N(u)$, the robber is trapped at $x$ and will eventually be captured by $C_2$, contradicting the assumption $c(G) \ge 3$. Hence, $N(x) \cap N(u) \neq \emptyset$, and by symmetry the same holds for every vertex of $H$.

Now, choose $p_1(C_1)=c$ and select $p_1(C_2) \in N(x) \cap N(u)$. The robber must then choose $p_1(R) \in N(u) \setminus N(c)$ to avoid immediate capture. In the next round, we examine five different cop moves, each of which forces the robber into a specific vertex and thereby reveals a neighbor of $p_1(R)$:
\begin{itemize}
    \item $p_2(C_1)=c$ and $p_2(C_2)=u$ force $p_2(R)=x$, hence $p_1(R) \in N(x)$.
    \item $p_2(C_1)=a$ and $p_2(C_2)=u$ force $p_2(R)=z$, hence $p_1(R) \in N(z)$.
    \item $p_2(C_1)=b$ and $p_2(C_2)=u$ force $p_2(R)=y$, hence $p_1(R) \in N(y)$.
    \item $p_2(C_1)=y$ and $p_2(C_2)=u$ force $p_2(R)=b$, hence $p_1(R) \in N(b)$.
    \item $p_2(C_1)=z$ and $p_2(C_2)=u$ force $p_2(R)=a$, hence $p_1(R) \in N(a)$.
\end{itemize}
Combining these implications, we obtain
\[
p_1(R) \in \bigl(N(u) \cap N(x) \cap N(y) \cap N(z) \cap N(a) \cap N(b)\bigr) \setminus N(c).
\]

By a completely symmetric argument (rotating the labels of the six vertices), we conclude that for every subset $F \subseteq V(H)$ of size $5$, there exists a vertex $v_F \in N(u)$ whose neighborhood within $H$ is exactly $F$; i.e., $N(v_F) \cap H = F$.

Hence, $d(x)=4+|N(x)\cap N(u)|\geq 9$, contradicts the degree condition $\Delta(G) \le 6$, completes the proof for the case $|E(S_1,S_2)| \ge 5$.





\subsection{Containing $B_6$}
Finally, suppose $H$ contains $B_6$. From the previous analysis, we may assume that $H$ cannot be covered by two cliques. Under this condition and $c(H) \ge 2$, the graph $H$ must be isomorphic either to $B_6$ itself or to the graph shown in Figure~\ref{fig:B6-variant}. 

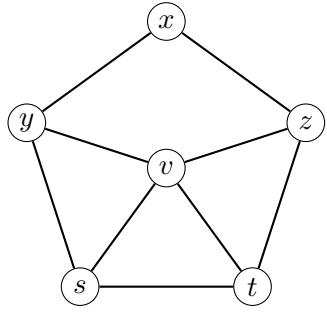
\begin{figure}[htbp]
    \centering
    \begin{tikzpicture}[
        vertex/.style = {draw, circle, inner sep=1.5pt, minimum size=5mm},
        edge/.style = {thick},
        scale=1.3
    ]
        \node[vertex] (v) at (0, 0) {$v$};
        \foreach \angle/\name in {90/x, 162/y, 234/s, 306/t, 18/z} {
            \node[vertex] (\name) at (\angle:1.5) {$\name$};
        }
        
        \draw[edge] (x) -- (y) -- (s) -- (t) -- (z) -- (x);
        
        \foreach \name in {y, s, t, z} {
            \draw[edge] (v) -- (\name);
        }
    \end{tikzpicture}
    \caption{A graph $H$ containing $B_6$, not coverable by two cliques, with $c(H) \ge 2$.}\label{fig:B6-variant}
\end{figure}

Now, we rule out this case. We first show that $N(x) \cap N(u) \neq \emptyset$. If this is false, we could place the cops as $p_1(C_1)=v$ and $p_1(C_2)=u$, forcing the robber to choose $p_1(R)=x$ to avoid immediate capture. Thereafter, keeping $C_1$ fixed at $v$ while $C_2$ moves along a shortest path towards $x$ would trap the robber at $x$ and lead to capture—a contradiction. Hence, $N(x) \cap N(u)$ is non‑empty.

Next, we prove that the set $A := \bigl(N(u) \cap N(x) \cap N(y) \cap N(z)\bigr) \setminus N(v)$ has size at least two. Suppose $|A| \le 1$. Choose $p_1(C_1)=v$ and select $p_1(C_2) \in N(u) \cap N(x)$ such that $A \subseteq \{p_1(C_2)\}$. We then examine three different cop strategies:
\begin{itemize}
    \item Set $p_2(C_1)=v$ and $p_2(C_2)=u$. This forces $p_2(R)=x$, implying $p_1(R) \in N(x)$.
    \item Set $p_2(C_1)=y$ and $p_2(C_2)=u$. Then, $p_2(R) \in \{z, t\}$. If $p_2(R)=t$, moving $p_3(C_1)=v$ and $p_3(C_2)=u$ captures the robber. Thus $p_2(R)=z$, and so $p_1(R) \in N(z)$.
    \item Set $p_2(C_1)=z$ and $p_2(C_2)=u$. Then, $p_2(R) \in \{y, s\}$. If $p_2(R)=s$, moving $p_3(C_1)=v$ and $p_3(C_2)=u$ again captures the robber. Thus $p_2(R)=y$, and so $p_1(R) \in N(y)$.
\end{itemize}
Combining these, we get $p_1(R) \in A$, which forces $p_1(R) = p_1(C_2)$—an immediate contradiction. Therefore, $|A| \ge 2$.

We now establish a second lower bound. Let $B := \bigl(N(u) \cap N(x) \cap N(z)\bigr) \setminus N(y)$. We claim that $|B| \ge 2$. Assume $|B| \le 1$. Place $p_1(C_1)=y$ and choose $p_1(C_2) \in N(z) \cap N(u)$ such that $B \subseteq \{p_1(C_2)\}$.
The robber must then choose $p_1(R) \in N(u)$; if $p_1(R)=t$, setting $p_2(C_1)=v$ and $p_2(C_2)=u$ would capture the robber. Now consider two cop moves:
\begin{itemize}
    \item Set $p_2(C_1)=y$ and $p_2(C_2)=u$. This forces $p_2(R) \in \{z, t\}$. The case $p_2(R)=t$ leads to capture as before, so $p_2(R)=z$, implying $p_1(R) \in N(z)$.
    \item Set $p_2(C_1)=v$ and $p_2(C_2)=u$. This forces $p_2(R)=x$, implying $p_1(R) \in N(x)$.
\end{itemize}
Thus $p_1(R) \in B$, forcing $p_1(R) = p_1(C_2)$, another contradiction. Hence $|B| \ge 2$.

Finally, observe that the sets $A$ and $B$ are both contained in $N(u) \cap N(x) \cap N(z)$ and are disjoint from each other. Consequently,
\[
|N(u) \cap N(x) \cap N(z)| \ge |A| + |B| \ge 2 + 2 = 4.
\]
Therefore,
\[
d(z) \ge |\{x, t, v\}| + |N(u) \cap N(x) \cap N(z)| \ge 3 + 4 = 7,
\]
which contradicts the degree condition $\Delta(G) \le 6$. This completes the proof for the case where $H$ contains $B_6$.

\section{Computer verification of the remaining cases}\label{computation}

It remains to consider the family $\mathcal{F}$ of graphs $G$ on $13$ vertices that satisfy condition ($\ast$) and the following additional constraint:

($\dagger$) For every vertex $v \in V(G)$ with $d(v)=6$, the graph $G - N[v]$ is isomorphic to either $B_6$ or $T_6$.

\subsection{Method}
Our goal is to prove that $\mathcal{F} = \emptyset$. A direct enumeration of all candidates is computationally prohibitive; therefore, we employ an indirect method based on a case analysis.

We first partition $\mathcal{F}$ (with possible overlaps) into five types, defined as follows:
\begin{itemize}
    \item \textbf{Type A}: There exist three vertices $u,v,w$ in $G$ with $d(u)=d(v)=d(w)=6$ and $(u,v),(v,w),(u,w) \notin E(G)$.
    \item \textbf{Type B}: There exist two vertices $u,v$ in $G$ with $d(u)=d(v)=6$, $(u,v) \notin E(G)$, and every vertex $w$ of degree $6$ satisfies $w \in N[u] \cup N[v]$.
    \item \textbf{Type C}: There exist three independent vertices $u,v,w$ in $G$ such that the induced subgraph $H = G - \{u,v,w\}$ satisfies $\Delta(H) \leq 4$.
    \item \textbf{Type D}: There exist three vertices $u,v,w$ in $G$ with $d(u)=d(v)=d(w)=6$ and $(u,v),(v,w),(u,w) \in E(G)$, and the induced subgraph $H = G - \{u,v,w\}$ satisfies $\Delta(H) \leq 5$.
\end{itemize}

\noindent We now show that every graph $G \in \mathcal{F}$ belongs to at least one of these types.

Let $G \in \mathcal{F}$. If there exist two distinct vertices $u, v$ with $d(u)=d(v)=6$ and $(u,v) \in E(G)$, choose a vertex $w \in V(G) \setminus (N[u] \cup N[v])$ having the maximum degree in the subgraph $G - (N[u] \cup N[v])$. If $d(w)=6$, then $G$ is of Type A; if $d(w) < 6$, then $G$ is of Type B.

If no such pair $u,v$ exists, then the set $\{x \in V(G): d(x)=6\}$ forms a clique (possibly empty). We consider its size:
\begin{itemize}
    \item If $|\{x : d(x)=6\}| = 0$, choose an independent set $\{u,v,w\}$ (possible because $c(G) \geq 3$). Since $G$ is $4K_1$-free, the graph $G$ is of type C.
    \item If $|\{x : d(x)=6\}| = 1$, say $d(u)=6$, choose $v,w$ such that $(u,v),(v,w),(u,w) \notin E(G)$ (again by $c(G) \geq 3$). As $G$ is $4K_1$-free, it is of type C.
    \item If $|\{x : d(x)=6\}| = 2$, say $\{u,v\}$. Since $N[u] \nsubseteq N[v]$, pick $w \in N(u) \setminus N[v]$. By $c(G) \geq 3$, there exists $t \in V(G) \setminus (N[v] \cup N[w])$. The independent set $\{v,w,t\}$ dominates $G$, and $|N(u) \setminus \{v,w,t\}| \leq 4$; thus $G$ is of type C.
    \item If $|\{x : d(x)=6\}| \geq 3$, then any three vertices of degree $6$ yield that $G$ is a graph of type D.
\end{itemize}

Having established that types A--D cover $\mathcal{F}$, we prove that each type is empty via computational enumeration.

\paragraph*{Algorithmic framework}
The algorithms for the four types follow a similar three‑step pattern with minor adaptations specific to each type.

\subsubsection*{Type A}
Let $H = G - \{u,v,w\}$, which has $10$ vertices. Since $G$ is $4K_1$-free, we have $\Delta(H) \leq 5$. Moreover,
\begin{align*}
    |E(H)| &= \frac12 \sum_{x \in V(H)} d_H(x) \\
           &\leq \frac12 \sum_{x \in V(H)} \bigl(d_G(x) - |N(x) \cap \{u,v,w\}|\bigr) \\
           &\leq \frac12 \Bigl( \sum_{x \in V(H)} 6 \;-\; d(u)-d(v)-d(w) \Bigr) = 21 .
\end{align*}
For any $x \in V(H)$ with $d_H(x)=5$, we have $d_G(x)=6$ (because $N(x) \cap \{u,v,w\} \neq \emptyset$). By ($\dagger$), $G-N[x]$ is isomorphic to $B_6$ or $T_6$; consequently, $H-N_H[x]$ is obtained from $B_6$ or $T_6$ by deleting two non‑adjacent vertices. Hence $H-N_H[x]$ is isomorphic to one of $K_3+K_1$, $2K_2$, or $P_4$. Therefore, $H$ satisfies:

($\star$) For every $x \in V(H)$ with $d_H(x)=5$, the graph $H-N_H[x]$ is isomorphic to $K_3+K_1$, $2K_2$, or $P_4$.

The algorithm proceeds as follows:
\begin{itemize}
    \item \textbf{Step 1.} Generate all graphs $H$ on $10$ vertices satisfying $|E(H)| \leq 21$, $\Delta(H) \leq 5$, $H$ is $4K_1$-free, and ($\star$) holds. Enumerate them up to isomorphism.
    \item \textbf{Step 2.} Introduce three new vertices $u,v,w$. Construct every $4K_1$-free graph $G$ on the vertex set $\{u,v,w\} \cup V(H)$ such that:
    \begin{itemize}
        \item $d_G(u)=d_G(v)=d_G(w)=6$,
        \item $(u,v),(v,w),(w,u) \notin E(G)$,
        \item $H$ is an induced subgraph of $G$.
    \end{itemize}
    Check whether $G$ satisfies both ($\ast$) and ($\dagger$), except the connectedness condition (because it only rules out a small number of graphs).
    \item \textbf{Step 3.} For each $G$ passing step~2, verify whether $c(G) \leq 2$.
\end{itemize}

\subsubsection*{Type B}
Let $H = G - \{u,v\}$ (so $|V(H)|=11$). Hence $\Delta(H) \leq 5$.
\begin{itemize}
    \item \textbf{Step 1.} Generate all $4K_1$-free graphs $H$ on $11$ vertices with $\Delta(H) \leq 5$, up to isomorphism.
    \item \textbf{Step 2.} Introduce two new vertices $u,v$. Construct every $4K_1$-free graph $G$ on $\{u,v\} \cup V(H)$ such that:
    \begin{itemize}
        \item $d_G(u)=d_G(v)=6$,
        \item $(u,v) \notin E(G)$,
        \item $H$ is an induced subgraph of $G$.
    \end{itemize}
    Check ($\ast$) and ($\dagger$) for $G$, except the connectedness condition.
    \item \textbf{Step 3.} Verify $c(G) \leq 2$ for each admissible $G$.
\end{itemize}

\subsubsection*{Type C}
Let $H = G - \{u,v,w\}$ ($|V(H)|=10$). By definition, $\Delta(H) \leq 4$.
\begin{itemize}
    \item \textbf{Step 1.} Generate all $4K_1$-free graphs $H$ on $10$ vertices with $\Delta(H) \leq 4$, up to isomorphism.
    \item \textbf{Step 2.} Introduce three new vertices $u,v,w$. Construct every $4K_1$-free graph $G$ on $\{u,v,w\} \cup V(H)$ such that:
    \begin{itemize}
        \item $d_G(u),d_G(v),d_G(w) \leq 6$,
        \item $(u,v),(v,w),(w,u) \notin E(G)$,
        \item $H$ is an induced subgraph of $G$.
    \end{itemize}
    Check ($\ast$) and ($\dagger$), except the connectedness condition.
    \item \textbf{Step 3.} Verify $c(G) \leq 2$.
\end{itemize}

\subsubsection*{Type D}
Let $H = G - \{u,v,w\}$ ($|V(H)|=10$). Here $\Delta(H) \leq 5$.
\begin{itemize}
    \item \textbf{Step 1.} Generate all $4K_1$-free graphs $H$ on $10$ vertices with $\Delta(H) \leq 5$, up to isomorphism.
    \item \textbf{Step 2.} Introduce three new vertices $u,v,w$. Construct every $4K_1$-free graph $G$ on $\{u,v,w\} \cup V(H)$ such that:
    \begin{itemize}
        \item $d_G(u)=d_G(v)=d_G(w)=6$,
        \item $(u,v),(v,w),(w,u) \in E(G)$,
        \item $H$ is an induced subgraph of $G$.
    \end{itemize}
    Check ($\ast$) and ($\dagger$).
    \item \textbf{Step 3.} Verify $c(G) \leq 2$.
\end{itemize}

\subsection{Computation results}

We present here selected data obtained from our computational enumeration and verification algorithms. For detail codes, see \url{https://github.com/trainboy1024/Codes-to-upload} . To achieve high computational efficiency while leveraging advanced graph libraries, we adopted a hybrid approach. The core enumeration framework was implemented in C++ for performance; and this framework interfaced with SageMath, utilizing its powerful graph theory libraries, including nauty \cite{mckay2014practical}. For the step of determining the cop number of a given graph, we directly utilized the function copnumberver() from the publicly available code of \cite{char2025counterexample}.

Type A: $|\{H\in\mathsf{Graph}(10)\cap\mathsf{Forb}(4K_1)\colon |E(H)|\leq 21, \Delta(H)\leq 5,(\star)\text{ holds for $H$}\}|=14764$; A total of \(322\) graphs \(G\) were verified under this type.

Type B: $|\{H\in\mathsf{Graph}(11)\cap\mathsf{Forb}(4K_1)\colon \Delta(H)\leq 5\}|=95420$; A total of \(8205\) graphs \(G\) were verified under this type.

Type C: $|\{H\in\mathsf{Graph}(10)\cap\mathsf{Forb}(4K_1)\colon \Delta(H)\leq 4\}|=783$; A total of \(361217\) graphs \(G\) were verified under this type.

Type D: $|\{H\in\mathsf{Graph}(10)\cap\mathsf{Forb}(4K_1)\colon \Delta(H)\leq 5\}|=66107$; A total of \(26729\) graphs \(G\) were verified under this type.

The reported graph counts are up to simple symmetries (e.g., permuting vertices \(u, v, w\)) but do \emph{not} perform a full isomorphism reduction; thus, graphs that are isomorphic but differ beyond these basic symmetries are counted separately.

\section*{Acknowledgments}

The author is grateful to the School of Mathematics and Statistics of Northeastern University at Qinhuangdao, as well as the peers who provided help throughout the whole process. The author also thanks his parents and family for their support, which has allowed him to devote himself to research without distraction.

\bibliographystyle{plain}
\bibliography{PAPER}

\end{document}